\documentclass[12pt]{amsart}
\usepackage[utf8x]{inputenc}

\usepackage{dsfont}
\usepackage{amsmath}
\usepackage{amssymb}
\usepackage{graphicx}
\usepackage{amssymb}
\usepackage{amsfonts}
\usepackage{soul}
\usepackage{mathpazo}
\usepackage{color}
\usepackage{yfonts}
\usepackage{paralist}
\usepackage{stmaryrd}
\usepackage{amsxtra}
\usepackage{latexsym,mathrsfs}

\usepackage{mathabx}

\usepackage{epsfig, enumerate}
\usepackage{color}
\usepackage{hyperref}

\newtheorem{theorem}{Theorem}
\newtheorem*{lemma*}{Claim}
\newtheorem{lemma}[theorem]{Lemma}

\newtheorem{Remark}[theorem]{Remark}

\numberwithin{theorem}{section}
\numberwithin{equation}{section}

\title[Rainbow Cliques]{More on Rainbow Cliques in Edge-Colored Graphs}

\author{Xiao-Chuan Liu}
\address[Liu]{Instituto de Matemática da Universidade Federal de Alagoas,
	Av. Lourival Melo Mota, S/N, Maceió, Brasil}
\email{lxc1984@gmail.com}

\author{Danni Peng}
\address[Peng]{Instituto de Matemática Pura e Aplicada,
Estrada Dona Castorina 110, Jardim Botânico, Rio de Janeiro, 22460-320, Brasil}
\email{dannipeng49@gmail.com}

\author{Xu Yang}
\address[Yang]{Instituto de Computação da Universidade Federal de Alagoas,
	Av. Lourival Melo Mota, S/N, Maceió, Brasil}
\email{yang@ic.ufal.br}

\begin{document}
\maketitle{}
\begin{abstract}
In an edge-colored graph $G$, a rainbow clique $K_k$ is a $k$-complete subgraph in which all the edges have distinct colors. Let $e(G)$ and $c(G)$ be the number of edges and colors in $G$, respectively. In this paper, we show that for any $\varepsilon>0$,
if $e(G)+c(G) \geq (1+\frac{k-3}{k-2}+2\varepsilon) {n\choose 2}$ and $k\geq 3$, then for sufficiently large $n$, the number of rainbow cliques $K_k$ in $G$ is $\Omega(n^k)$.

We also characterize the extremal graphs $G$ without a rainbow clique $K_k$, for $k=4,5$, when $e(G)+c(G)$ is maximum.

Our results not only address existing questions but also complete the findings of Ehard and Mohr (Ehard and Mohr, Rainbow triangles and cliques in edge-colored graphs. {\it European Journal of Combinatorics, 84:103037,2020}).

\end{abstract}

\section{Introduction}
Mantel's Theorem (see~\cite{mantel1907vraagstuk},1907) and the subsequent renowned Tur\'an's Theorem (see~\cite{turan1941extremalaufgabe}, 1941) initiated the field of extremal graph theory. Tur\'an's Theorem states that for any graph $G$ with $n$ vertices, if the number of edges exceeds the extremal number $t_{n,k-1}$ (explained in the following paragraphs), the graph must contain a clique $K_k$ as a subgraph. Notably, Erd\H{o}s and Simonovits (in~\cite{erdHos1983supersaturated}, 1983) later observed that if a graph contains $\Omega (n^2)$ more edges than the aforementioned extremal number, then the number of clique $K_k$ contained in $G$ is at least $\Omega (n^2)$. This insight was further extended to encompass all graphs and hypergraphs. This crucial discovery, later dubbed the “supersaturation phenomenon" has significantly influenced contemporary research in extremal combinatorics. For instance, the recent breakthrough known as the  "Hypergraph Containers" method finds its applications mainly when the supersaturation results are applicable. For a deeper understanding, consult the papers and  surveys~\cite{saxton2015hypergraph},~\cite{balogh2015independent}, and~\cite{balogh2018method}.

In this paper, we direct our attention to the colored version of extremal problems.
Consider an edge-colored graph $G$, where we denote  $e(G)$ and $c(G)$ the number of edges and colors in $G$, respectively. A {\it rainbow clique $K_k$} within $G$ refers to $k$-complete subgraph where all the edge colors are distinct. In 2014, Li, Ning, Xu and Zhang \cite{li2014rainbow} established specific sufficient conditions for the presence of a rainbow triangle in terms of the sum of the number of colors and edges within the host graph $G$. One of these conditions is stated as follows.
\begin{theorem}\label{rainbow_triangle}(\cite{li2014rainbow})
Let $G$ be a colored graph on $n$ vertices satisfying $e(G)+c(G)\geq {n \choose 2}+n$. Then $G$ contains a rainbow triangle.
\end{theorem}

 Later in 2016, Xu, Hu, Wang and Zhang \cite{xu2016rainbow} extended the condition for the existence of a rainbow clique $K_k$, where $k\geq 4$. Let $T_{n,k}$ represent the Turán graph, defined as a complete $k$-partite graph with $n$ vertices, ensuring the parts have the most equal sizes possible. This construction guarantees a unique graph up to isomorphism. We denote $t_{n,k}$ as the number of edges of $T_{n,k}$, also known as the  Turán number. The generalized condition can be concisely expressed as follows:

\begin{theorem}\label{ExistClique}(\cite{xu2016rainbow})
Let $G$ be an edge-colored graph on $n$ vertices satisfying $e(G)+c(G)\geq {n\choose 2}+t_{n,k-2}+2$, then $G$ contains a rainbow clique $K_k$, where $k\geq 4$.
\end{theorem}

It is natural to consider the subsequent problem of demonstrating an abundance of rainbow cliques in the spirit of Erd\H{o}s and Simonovits.
Our first result, stated below, offers a supersaturation theorem in this context.
 Assuming the common criterion of substantial  $e(G)+c(G)$ values, the outcome showcases the existence of $\Omega(n^k)$ rainbow cliques.


\begin{theorem} \label{main1}
Let $G$ be an edge-colored graph on $n$ vertices.  For any $\varepsilon>0$, there exists some $\delta>0$ such that if $e(G)+c(G) \geq (1+\frac{k-3}{k-2}+2\varepsilon) {n\choose 2}$ and $k\geq 3$, then for sufficiently large $n$, $G$ contains at least $\delta n^k$ rainbow cliques $K_k$. 
\end{theorem}
This result demonstrates a notable enhancement over  Theorem 4 of the paper~\cite{li2021counting}, which specifically addressed the case of $k=3$. Additionally, it is worth noting that ~\cite{li2021counting} required a more stringent condition, namely, a large minimum color degree of the graph.
 Our proof introduces a modification of the classical method known as the Varnavides averaging argument (refer to~ \cite{varnavides1959certain}).

It is also possible to describe extremal colored graphs for this extremal problem. The paper~\cite{fujita2019sufficient} offers insights into characterizing extremal graphs $G$ that lack rainbow triangles when $e(G)+c(G)$ reaches its maximum. Additionally, in~\cite{ehard2020rainbow}, the authors characterized the edge-colored graphs without rainbow $K_k$ for $k\geq 6$ that maximize the sum of edge number and color number. Specifically, if a graph $G$ on $n$ vertices satisfies $e(G)+c(G)={n\choose 2}+t_{n,k-2}+1$ and does not contain a rainbow $K_k$, then $G$ is complete and contains a rainbow $T_{n,k-2}$.
The graph $G$ can be obtained as follows, up to isomorphism.
\begin{enumerate}
\item Firstly, divide $n$ vertices into $k-2$ parts as evenly as possible.
\item Then, assign pairwise distinct colors to all the edges between different parts, and assign all the edges within each part a new color.
\end{enumerate}
This construction yields the desired extremal graphs under the given conditions.

In the same paper~\cite{ehard2020rainbow}, the cases $k=4$ and $k=5$ were also considered, but were regarded as having "little hope" for arriving at a nice characterization.

 In this paper, we complete this puzzle by characterizing the extremal graphs that do not have a rainbow clique $K_k$, for $k=4$ and 5. We demonstrate that when $n$ is slightly larger than $k$, the extremal graphs without a rainbow $K_4$ (or $K_5$) are also complete and have the same coloring pattern with the ones without rainbow $K_k$, $k\geq 6$.

\begin{theorem}\label{main2}
Let $G$ be an edge-colored graph with $n$ vertices. If $e(G)+c(G)={n\choose 2}+t_{n,k-2}+1$, where either $k=4$ and $n\geq 8$, or $k=5$ and $n\geq 9$, and $G$ does not have a rainbow $K_k$, then $G$ satisfies the following conditions:
\begin{enumerate}
\item $G$ is complete and
\item $G$ has a rainbow $T_{n,k-2}$ and all the edges within every part are colored alike with a new color.
\end{enumerate}
\end{theorem}
Our result is sharp in the following sense. One of the reasons why the cases $k=4$ and $k=5$ are considered difficult in the paper ~\cite{ehard2020rainbow} is that counterexamples can be found that do not satisfy the pattern described above. For instance, in~\cite{ehard2020rainbow}, an example for the case $k=5$ and $n=8$ was provided  (see Figure 3 in their paper). In Section 2 below, 
We will further provide an example of
$k=4$ and $n=7$. However, surprisingly, our theorem above shows that these counterexamples are the best that can be obtained.
\section{Preliminary}

In this paper, we focus exclusively on simple and undirected graphs. The edge coloring of graph $G$ is not necessarily proper. For any edge $e\in E(G)$, we define $c(e)=c$ to be the color of $e$. We define a color $c$ to be {\it saturated} at a vertex $v$ if there exists an edge $e$ adjacent to $v$ with $c(e)=c$, and the color $c$ does not appear in $G-v$.  The saturated degree of vertex $v$, denoted as $d_G^s(v)$, represents the number of colors that are saturated at $v$ in graph $G$. When there is no confusion, we can simply write $d^s(v)$ instead of $d_G^s(v)$. Let $d(v)$ be the degree of vertex $v$. It is evident that $d^s(v)\leq d(v)$ and $\sum_{v\in V(G)}d^s(v)\leq 2c(G)$.
We provide some estimates regarding the Turán number $t_{n,k}$ which will be needed in the proofs. Let $n=pk+i$, where $0\leq i\leq k-1$.
\begin{eqnarray}
t_{n,k}=\frac{(k-1)(n^2-i^2)}{2k} +\binom i2 \\
\frac{(k-1)n^2}{2k}-\frac{n}{4}\leq
t_{n,k}\leq \lfloor \frac{(k-1)n^2}{2k}\rfloor\\
t_{n+1,k}-t_{n.k}=n-\frac{n-i}{k}
\end{eqnarray}
If $G$ is a complete graph edge-colored by $t_{n,k-1}+1$ colors and does not contain a rainbow $K_k$,
according to Theorem~\ref{ExistClique}, we know that for any $v\in V(G)$,
$c(G-v)\leq t_{n-1,k-2}+1$ and $d^s(v)\geq t_{n,k-2}-t_{n-1,k-2}$.
In the case where $n$ is divisible by $k-2$,
the bounds of the summation of saturated degrees in $G$ become special.
Since $n(t_{n,k-2}-t_{n-1,k-2})=(k-3)n^2/(k-2)=2t_{n,k-2}$, we can deduce that
\begin{align}\label{sumofds2}
	& 2t_{n,k-1}+2=2c(G)\geq \sum_{v\in V(G)}d^s(v)
	\geq n(t_{n,k-2}-t_{n-1,k-2})
	=  2t_{n,k-2}.
\end{align}
We can also consider an alternative method for calculating the summation of saturated degrees in $G$. Let $c_t$ be the number of colors that saturated by $t$ vertices, for $t=0,1,2$. Then $2c_2+c_1=\sum_{v\in V(G)}d^s(v)$.

The first lemma we use several times is a property directly from the argument in Lemma 11 of \cite{ehard2020rainbow}.

\begin{lemma}(\cite{ehard2020rainbow}, Lemma 11)\label{ResultFromLemma11}
Let $G$ be an edge-colored graph on $n$ vertices such that $e(G)+c(G)={n\choose 2}+t_{n,k-2}+1$ and $G$ does not contain a rainbow $K_k$, for $n> k\geq 4$, then there is a vertex $v$ in $G$ such that $e(G-v)+c(G-v)={n-1\choose 2} +t_{n-1,k-2}+1$.
\end{lemma}

The following lemma can be found in \cite{fujita2019sufficient}, although here we consider a slightly different scenario. The proof is provided in Appendix.

\begin{lemma}\label{noK3}
If $e(G)+c(G)=\frac{n(n+1)}{2}-1$ and $G$ does not contain a rainbow triangle, monochromatic $C_3$ or $P_4$, then $G$ is a complete graph $K_n$ and $c(ij)=i$, where $ij=ji\in E(G)$, with $i<j$.
\end{lemma}

\begin{Remark}
Note that if $c(K_4)=3$, there are only two possibilities of an edge-colored $K_4$ with no rainbow triangle. One possibility is  as described in Lemma \ref{noK3}. The other possibility is that there is a monochromatic $C_4$, and the other two edges are colored with two new distinct colors.
\end{Remark}





\begin{lemma}\label{K6-}
Let $K_6^-$ be an edge-colored graph obtained by deleting an edge from $K_6$. Let the vertex set be $V_1\cup V_2$, where $V_1=\{v_1, v_2\}$ contains the two non-adjacent vertices with degree 4, and $K_6^-[V_2]=K_4$. If $c(K_6^-)=11$ and $K_6^-$ has no rainbow $K_4$, then there is rainbow $K_{2,4}$ in $K_6^-$, and either $K_6^-[V_2]$ has neither a rainbow triangle nor a monochromatic $C_3$ and $P_4$, or $K_6^-[V_2]$ has a monochromatic $C_4$.
\end{lemma}
\begin{proof}
Since $K_6^-$ does not contain a rainbow $K_4$, if follows that for any vertex $v$ in $K_6^-$, $d(v)+d^s(v)\geq 8$. This is because if $e(K_6^--v)+c(K_6^--v)\geq 18$ and by Theorem \ref{ExistClique}, there would be a rainbow $K_4$ in $K_6^--v$. By simple computation, for the vertices $v_1$ and $v_2$ with $d(v_1)=d(v_2)=4$, we have $d^s(v_1)=d^s(v_2)=4$. Therefore, we obtain a rainbow $K_{2,4}$ in $K_6^-$. Regarding $K_6^-[V_2]$, we have that $K_6^-[V_2]=K_4$ and $K_6^-[V_2]$ has three new colors other than the eight colors in the rainbow $K_{2,4}$. It is important to note that there is no rainbow triangle with at least two new colors in $K_6^-[V_2]$, because if there were,  there would be a rainbow $K_4$ in $K_6^-$. Consequently, there is no any rainbow triangle in $K_6^-[V_2]$. Based on the previous argument, there are only two ways of coloring $K_4$ such that there is no rainbow triangle.
\end{proof}

\begin{lemma}\label{K6_two_ways}
If $K_6$ is edge-colored with $10$ colors and does not contain a rainbow $K_4$, then either there is a rainbow $T_{6,2}$ as a subgraph, or there is a monochromatic $C_6$ as a subgraph.
\end{lemma}
\begin{proof}
By the discussion in the beginning of this section, we have that
\begin{equation}
	c_2+c_1+c_0=10\text{\ \ \ and\ \ \  } 18\leq 2c_2+c_1\leq 20
\end{equation}
The possible solutions of the above are
\begin{equation}
	(c_2,c_1,c_0)=(10,0,0), (9,1,0), (9,0,1), (8,2,0)
\end{equation}
Firstly, $(10,0,0)$ can be eliminated, because that means all the colors only appear once. Similarly, $(9,1,0)$ ($(8,2,0)$) can be eliminated, because if we delete one appropriate vertex (two vertices), then the remaining graph has a rainbow $K_4$. So there is only one solution which is $(9,0,1)$, that is 9 colors only appear once. On the other side, for any vertex $v\in K_6$, $d(v)+d^s(v)\geq 8$. If we delete the edges whose color only appears once, then every vertex in the remaining graph has degree at least 2. Therefore there are only two ways, either there are two monochromatic $C_3$ and $K_6$ has a rainbow $T_{6,2}$, or there is a monochromatic $C_6$.
\end{proof}








\begin{lemma}\label{BaseCaseLemma11a}
Let $G$ be an edge-colored graph on 8 vertices. If $e(G)+c(G)=45$ and $G$ does not contain a rainbow $K_4$, then $G$ is complete.
\end{lemma}
\begin{proof}
To prove that $G$ is  complete, we will eliminate the following cases and conclude that every vertex of $G$ has degree $7$.

{\bf Case 1}. there is a vertex $d(v)\leq 5$, hence  $d^s(v)\leq 5$.

We have $e(G-v)+c(G-v)\geq 35$. By Theorem \ref{ExistClique}, there is a rainbow $K_4$ in $G-v$.  Therefore all the vertices in $G$ have a degree of at least 6.

{\bf Case 2}. there is a vertex $v$ such that $d(v)=6$ and $d^s(v)=6$.

In this case, let $u$ be the vertex in $G$ that is not adjacent to $v$. Then we have $e(G-\{u,v\})+c(G-\{u,v\})\geq 45-12-12=21$, which indicates that there is a rainbow triangle in $G-\{u,v\}$ according to Theorem \ref{rainbow_triangle}. Therefore, when we consider $v$ along with it, there exists a rainbow $K_4$.

{\bf Case 3}. there is a vertex $d(v)=6$ and $d^s(v)\leq 4$.

In this case, we have $e(G-v)+c(G-v)\geq 35$, which indicates that there is a rainbow $K_4$ in $G-v$ according to Theorem \ref{ExistClique}.

{\bf Case 4}. there is a vertex $d(v)=6$, and $d^s(v)=5$.

In this case, we observe that $e(G-v)+c(G-v)=34$. Furthermore, there exists a vertex $u$ in $G$ that is not adjacent to $v$, with $d(u)=6$ and $d^s(u)=5$. When considering $G-v$, we have $d_{G-v}(u)=6$ and $d_{G-v}^s(u)\geq 5$.

If $d_{G-v}^s(u)=6$, in $G-\{u,v\}$, we have $e(G-\{u,v\})+c(G-\{u,v\})=22\geq 21$. This implies that there exists a rainbow triangle in $G-\{u,v\}$, and consequently, there is a rainbow $K_4$ in $G-v$.

If $d_{G-v}^s(u)=5$, then $e(G-\{u,v\})+c(G-\{u,v\})=23$. It is important to note that $G-\{u,v\}$ is a graph with 6 vertices and cannot be isomorphic to $K_6$. Alternatively, if $G-\{u,v\}$ is isomorphic to $K_6$, then $G-v$ would be isomorphic to $K_7$ and we would have $e(G-v)+c(G-v)=34$. By Lemma \ref{ResultFromLemma11}, there exists a vertex $w$ such that $e(G-\{v,w\})+c(G-\{v,w\})={6\choose 2}+t_{6,2}+1=25$, with $d_{G-v}(w)+d^s_{G-v}(w)=9$ implying $u\neq w$. According to Lemma \ref{K6_two_ways}, $K_6$ has two possible ways of coloring. In neither of these coloring schemes, we can find a vertex with $d^s_{G-\{v,w\}}(v')\geq 4$. This means that there is no vertex in $G-\{v,w\}$ that can be considered as $u$.

If $G-\{u,v\}$ is isomorphic to $K_6^-$, and thus $c(G-\{u,v\})=9$, we can add an edge in $G-\{u,v\}$ along with a new color. The resulting graph $G'$ satisfies $e(G')+c(G')=25$. According to Lemma \ref{K6_two_ways}, there are two ways to color $G'$. By deleting any edge whose color is unique in both ways of coloring, we obtain two distinct ways of coloring $G-\{u,v\}$. In either of the colored graph $G-\{u,v\}$, if we add back either $u$ or $v$, it can be easily verified that there exists a rainbow $K_4$ in $G-\{u\}$ or $G-\{u\}$.

If $G-\{u,v\}$ is isomorphic to $K_6$ with two missing edges, following a similar argument as above, we can add a new edge with a new color to $G-\{u,v\}$ and we obtain a new graph $G'$ that is isomorphic to $K_6^{-}$ and has $c(G')=11$. According to Lemma \ref{K6-}, there are two ways to color $G'$. It is straightforward to verify that by deleting an edge with a unique color and adding back either $u$ or $v$, there must be a rainbow $K_4$.

If $G-\{u,v\}$ is isomorphic to $K_6$ with three missing edges, we add back two edges with the same colors that already exist in $G-\{u,v\}$. We denote the resulting graph as $G'$. In $G'$, we delete two non-adjacent edges $e_1$ and $e_2$ such that $c(G')=c(G'-\{e_1,e_2\})=11$. It is straightforward to verify that $(G'-\{e_1,e_2\})\cup \{v\}$ must have a rainbow $K_4$.

Based on Case 1 mentioned earlier, $G-\{u,v\}$ cannot be isomorphic to $K_6$ with more than $3$ edges removed. Therefore, we can conclude the proof.
\end{proof}

\begin{Remark}\label{CounterExample}
In the case of $G$ being an edge-colored graph with 7 vertices, and satisfying $e(G)+c(G)=34$, we can provide an example where $G$ does not contain a rainbow $K_4$ and is not complete. Let $u$ and $v$ be two non-adjacent vertices with $d(u)=d(v)=d^s(u)=d^s(v)=5$. The graph $G-\{u,v\}$ is an edge-colored graph with no rainbow triangle, as shown in the example provided in Lemma \ref{noK3}.
\end{Remark}

\begin{lemma}\label{BaseCaseLemma11b}
Let $G$ be an edge-colored graph with 9 vertices. If $e(G)+c(G)=64$ and $G$ does not contain a rainbow $K_5$, then $G$ is complete.
\end{lemma}

\begin{proof}
We follow the same idea as the proof of Lemma \ref{BaseCaseLemma11a}. It is worth noting that for a graph $H$ with 8 vertices, if $e(H)+c(H)\geq 51$, then $H$ must contain a rainbow $K_5$. Similarly, if $e(H)+c(H)\geq 46$, then $H$ must contain a rainbow $K_4$.

{\bf Case 1}. There is a vertex $d(v)\leq 6$, implying $d^s(v)\leq 6$.

In this case, we observe that $e(G-v)+c(G-v)\geq 52$. By Theorem \ref{ExistClique}, there is a rainbow $K_5$ in $G-v$.

Consequently, all vertices in $G$ have degrees of at least $7$.

{\bf Case 2}. There is a vertex $v$ with $d(v)=7$ and $d^s(v)\leq 6$.
In this case, we similarly observe that $e(G-v)+c(G-v)\geq 51$. By Theorem \ref{ExistClique}, there is a rainbow $K_5$ in $G-v$.

{\bf Case 3}. There is a vertex $v$ with $d(v)=7$ and $d^s(v)=7$.

In this case, we find that $e(G-v)+c(G-v)=50$. According to  Theorem \ref{ExistClique}, there is a rainbow $K_4$ in $G-v$. Therefore, together with vertex $v$, there exists a rainbow $K_5$.

By eliminating all the cases mentioned above, we can  conclude that $G$ is complete.

\end{proof}

\begin{lemma}\label{BaseCaseLemma10a}
If $K_8$ is an edge-colored complete graph with $17$ colors and does not contain a rainbow $K_4$, then it must contain a rainbow $T_{8,2}$ as a subgraph.
\end{lemma}
\begin{proof}
Assuming that $G=K_8$ does not contains a rainbow $K_4$ and is edge-colored by $t_{8,2}+1=17$ colors, we can apply the  inequality~\ref{sumofds2}, which gives us
\begin{align*}
c_2+c_1+c_0 =17,
32\leq 2c_2+c_1\leq 34.
\end{align*}
The solutions to the above inequalities are as follows:
\begin{align*}
(c_2,c_1,c_0)=(17,0,0),(16,1,0),(16,0,1),(15,2,0).
\end{align*}
If $c_1\leq 3$ and $c_0=0$, we can delete three appropriate vertices, and $G$ will still contain a rainbow $K_4$. As a result, the only feasible solution for the values of $(c_2,c_1,c_0)$ is $(16,0,1)$. In this solution, out of 17 colors, 16 colors appear only once, and all the other edges share with a new color, denoted as color 0. According to  (\ref{sumofds2}), we have $d^s(v)=4$ for every vertex $v\in V(G)$, meaning that each vertex is incident to 4 edges uniquely colored. Out of these, 3 edges have the new color 0.  
For convenience,  let's consider  the subgraph $H$ of $G$ whose edges are those with color 0 in $G$, and $V(H)=V(G)$. It is important to note that $d_H(v)=3$ for any $v\in V(H)$. 
Since $G$ dose not contain a rainbow $K_4$, we conclude that
$e(H[T])\geq 2$ for each $4$-tuple $T$ of vertices in $V(H)$.
Let $X=\{x_1,x_2,x_3,x_4\}=\{x_1\}\cup N(x_1)$ represent the subset of $V(H)$, and let $Y=\{y_1,y_2,y_3,y_4\}$ be the complement of $X$ in $V(H)$.
Suppose that $H[X]$ has $t(\geq 3)$ edges. In this case, $H[Y]$ also has $t$ edges, as every vertex $v\in V(G)$ is incident to three edges colored with $0$.

We assert that $t=6$. To support this claim, we examine the following cases:
\begin{enumerate}
\item
If $t=5$, we assume that $x_2x_3\notin E(H)$. As a consequence, we have  $N(x_4)=\{x_1,x_2,x_3\}$.
For any $i,j$, when we examine the  $4$-tuple $\{x_1,x_4,y_i,y_j\}$, it becomes evident that $y_iy_j\in E(H)$.
This implication leads us to the conclusion that $H[Y]$ contains 6 edges, which ultimately results in a contradiction;
\item
If $t=4$, we assume that $x_2x_3,x_3x_4\notin E(H)$. Then $x_3$ is adjacent to two vertices in $Y$. For the sake of simplicity, let's assume that $x_3y_3$ and $x_3y_4$ are edges in $H$. By considering the $4$-tuple $\{x_1,x_3,y_1,y_2\}$, we observe that  $y_1y_2\in E(H)$. Upon examining all  $4$-tuples $\{x_1,y_i,y_j,y_k\}$ for all $i,j,k$, it becomes clear that there are 2 edges in $\{y_i, y_j, y_k\}$. Consequently, $H[Y]$ forms a cycle  $C_4$. Without loss of generality, let the edges be $y_1y_2$, $y_2y_3$, $y_3y_4$, $y_1y_4$. 
Examining the $4$-tuple $\{x_3,x_4,y_1,y_2\}$, we deduce that one of the edges $x_4y_1$ or $x_4y_2$ must be present in $E(H)$; let's assume it is  $x_4y_1$ without loss of generality.
However, upon considering the subgraph $H[x_3,x_4,y_1,y_3]$, we observe that it contains only a single edge, which leads to a contradiction;
\item
If $t=3$, each of the vertices $x_2$, $x_3$ and $x_4$ is adjacent to two vertices in $Y$.
Without loss of generality, let's assume that $x_2y_1$ and $x_2y_2$ are edges in $H$. Considering the 4-tuple $\{x_1,x_2,y_3,y_4\}$, we discover that $y_3y_4\in E(H)$.  
Further investigation of the 4-tuple $\{x_1,y_1,y_2,y_3\}$ reveals that the set  $\{y_1,y_2,y_3\}$ contains at least two edges.
Therefore, due to the constraint of $t=3$, we deduce that $y_1y_4$ and $y_2y_4$ are not present in $E(H)$.  However, when considering the subgraph 
$H[x_1,y_1,y_2,y_4]$, it becomes apparent that it can have at most one edge, leading to a contradiction.
\end{enumerate}
From the above cases, we can conclude that $H$ is a disjoint union of two copies of $K_4$. Returning to the original graph $G$, we find that the remaining edges, which are not colored with $0$, form a rainbow $T_{8,2}$.
\end{proof}

\begin{lemma}\label{BaseCaseLemma10b}
If $K_9$ is edge-colored with $28$ colors and does not contain a rainbow $K_5$, then $K_9$ contains a rainbow $T_{9,3}$ as a subgraph.
\end{lemma}
\begin{proof}
Assuming $G=K_9$ does not contains a  rainbow $K_5$ and is edge-colored by $t_{9,3}+1=28$ colors,
by inequality~\ref{sumofds2}, we can observe the following inequalities: 
\begin{align*}
c_2+c_1+c_0=28,
54\leq 2c_2+c_1\leq 56.
\end{align*}
The solutions to these inequalities are:
\begin{align*}
(c_2,c_1,c_0)=(28,0,0),(27,1,0),(27,0,1),(26,2,0).
\end{align*}
If $c_1\leq 4$ and $c_0=0$, we can delete appropriate vertices and still retain a rainbow $K_5$ in $G$. 
There is only one feasible solution $(27,0,1)$ for $(c_2,c_1,c_0)$, as shown in the proof of Lemma \ref{BaseCaseLemma10a}. This feasible solution implies that $d^s(v)=6$ for all $v\in V(G)$. Let's assume that the color not saturated by any vertex is color $0$. Consider a subgraph $H$ of $G$ with the same vertex set and edges colored with 0. Note that $d_H(v)=2$ for any $v\in V(H)$
Since $G$ contains no rainbow $K_5$, it follows that $e(H[T])\geq 2$ for each $5$-tuple $T$ of $V(H)$.
$H$ is the union of disjoint cycles, as  every $v\in V(G)$ incident to exactly two edges colored with $0$. We will now  eliminate the following cases:
\begin{enumerate}
\item
if $H=C_9=v_1v_2v_3v_4v_5v_6v_7v_8v_9v_1$, then $H[v_1,v_2,v_4,v_6,v_8]$ has only one edge $v_1v_2$;
\item
if $H=C_6\cup C_3=v_1v_2v_3v_4v_5v_6v_1\cup v_7v_8v_9v_7$, then $H[v_1,v_3,v_5,v_6,v_8]$ has only one edge $v_5v_6$;
\item
if $H=C_5\cup C_4=v_1v_2v_3v_4v_5v_1\cup v_6v_7v_8v_9v_6$, then $H[v_1,v_3,v_4,v_6,v_8]$ has only one edge $v_3v_4$.
\end{enumerate}
Thus $H$ is the union of three disjoint cycles $C_3$.
Back to the original graph, we obtain a rainbow $T_{9,3}$ in $G$.
\end{proof}

\section{Counting Rainbow Cliques}

In this section we prove Theorem \ref{main1}.  Let $V_1$ and $V_2$ denote two sets of vertices in $G$. We define the bipartite graph with bipartition $V_1$ and $V_2$ as $G[V_1,V_2]$, and the induced subgraph with vertex set $V_1\cup V_2$ as $G[V_1\cup V_2]$. Additionally, $e(V_1,V_2)$ and $c(V_1,V_2)$ represent the number of edges and colors in $G[V_1,V_2]$, respectively. 

\begin{proof} Fix $\varepsilon >0$. For any subgraph $F$ of $G$ with $m$ vertices, where $m<n$ is also sufficiently large, if $e(F)+c(F)\geq (1+\frac{k-3}{k-2}+\varepsilon){m\choose 2}$, then by Theorem \ref{ExistClique}, $F$ contains a rainbow clique. Then we consider all $m$-vertex spanning subgraphs $M$ in $G$. The subgraphs that satisfy $e(M)+c(M)\geq (1+\frac{k-3}{k-2}+\varepsilon){m\choose 2}$ and provide a rainbow clique are referred to as the {\it good} ones. Let $\eta\in (0,1)$ be a value such that there are $\eta{n\choose m}$ good ones, while the non-good ones are $(1-\eta){n\choose m}$.

Now we consider $\sum_{|M|=m}(e(M)+c(M))$. Firstly, since we count every edge $e\in M$ for ${n-2 \choose m-2}$ times, so we have
\begin{equation}
\sum_{|M|=m}e(M)=e(G){n-2\choose m-2}.
\end{equation}
Secondly, each color appeared in $\sum_{|M|=m}c(M)$ is counted for at least ${n-2\choose m-2}$ times. So we have
\begin{equation}
\sum_{|M|=m}c(M)\geq c(G){n-2\choose m-2}.
\end{equation}
Therefore, we have
\begin{eqnarray}
\sum_{|M|=m}(e(M)+c(M))&\geq & (e(G)+c(G)){n-2\choose m-2}.\label{eMcM}\\
&=& (1+\frac{k-3}{2(k-2)}+2\varepsilon){n\choose 2}{n-2\choose m-2}
\end{eqnarray}
The left side is at most
\begin{equation}
\eta{n\choose m}\cdot 2\cdot {m\choose 2} + (1-\eta) {n\choose m}{m\choose 2}(1+\frac{k-3}{k-2}+\varepsilon)
\end{equation}
So
\begin{eqnarray}
1+\frac{k-3}{k-2}+2\varepsilon &\leq &  2\eta+(1-\eta)(1+\frac{k-3}{k-2}+\varepsilon)\\
& = & 2\eta +(1+\frac{k-3}{k-2}+\varepsilon)-\eta(1+\frac{k-3}{k-2}+\varepsilon)
\end{eqnarray}
Therefore, $\eta\geq \frac{\varepsilon}{1-\frac{k-3}{k-2}-\varepsilon}$. On the other hand, each rainbow clique $K_k$ is contained in at most ${n-k\choose m-k}$ many $m$-sets. Therefore, there exists a $\delta$ such that the number of rainbow clique $K_k$ is at least
\begin{equation}
\frac{\eta {n\choose m}}{{n-k \choose m-k}}=\frac{\eta {n\choose k}}{{m \choose k}}=\delta n^k
\end{equation}
\end{proof}

\begin{Remark}
In \cite{xu2020properly}, the authors give a condition of the existence of properly colored $C_4$ in terms of $e(G)+c(G)$. That is, if $e(G)+c(G)\geq {n\choose 2}+n+1$, then $G$ contains a properly colored $C_4$. We can use the same technique in the above proof to count the number of properly colored $C_4$ ($\Omega(n^3)$), when $e(G)+c(G)\geq (1+\varepsilon){n\choose 2}$, where $\varepsilon>0$.
\end{Remark}

The following result improves Proposition 12 in \cite{ehard2020rainbow} for $k\geq 5$ and $\ell=2$. It becomes evident that for $k\geq 5$, the likelihood of having exactly one rainbow clique $K_k$ is extremely low. 
\begin{theorem}
Let $G$ be an edge-colored graph on $n$ vertices
and $e(G)+c(G)\geq \binom n2 +t_{n,k-2}+2$.
If $n> k\geq 6$, or $k= 5$ and $n\geq 10$,
then $G$ contains at least two rainbow cliques $K_k$.
\end{theorem}

\begin{proof}
Let's suppose that there is exactly one rainbow $K_k$ in the graph.
Consider $n=s(k-2)+i$, where $0\leq i\leq k-3$. After some straightforward calculations, we obtain the following inequality:
\begin{align}\label{e+c}
e(G)+c(G)\geq
&\binom n2 +\frac{(k-3)(n^2-i^2)}{2(k-2)} +\binom i2 +2=n^2-\frac{(1+s)(n+i)}2+2
\end{align}
Let $A$ denote the vertex set of the rainbow clique $K_k$ and $B=V(G)\setminus A$.
We define $c(A,T)$ to represent the number of new colors of edges between set $A$ and $T\subseteq B$, which are distinct from those in the rainbow $K_k$.
It follows that $c(A,u)\leq k-2$ for any $u\in B$, as exceeding this value would result in the existence of another rainbow $K_k$. We can further divide this into two into 2 cases.

{\bf Case 1. } $n<2k$.

Note that in this case, $s<3$. To see why, suppose the contrary, then we would have: 
\begin{equation}
	s(k-2)+i<2k
\end{equation}
This implies $k< \frac{2s}{s-2}\leq 6$, which leads to a contradiction.

For the cases when $s=1$ or $s=2$, we can sum up the counts of edges and colors in $G[A]$, $G[B]$ and $G[A,B]$, resulting in:
\begin{align}
e(G)+c(G)
\leq & e(A)+c(A)+e(A,B)+c(A,B)+e(B)+c(B) \\
\leq & 2\binom k2 +k(n-k)+\sum_{u\in B}c(A,u) +2\binom {n-k}2 \\
\leq & 2\binom k2 +(2k-2)(n-k)+2\binom {n-k}2 \\
= & n^2-3n+2k.
\end{align}
Combining this with the lower bound~\ref{e+c} and substituting $n=s(k-2)+i$,
$$
(s^2-5s+4)k \geq 2s^2-10s-2is+4i+4.
$$
For the case of $s=1$, we have $i\leq 2$. However, since $n>k$, we must have $i\geq 3$, leading to a  contradiction.
For the case of $s=2$, we find:
\begin{align}
k\leq \frac{2s^2-10s-2is+4i+4}{s^2-5s+4}
= 4,
\end{align}
which contradicts the fact that $k\geq 5$.

{\bf Case 2. $n\geq 2k$.}

Note that in this case, we have $s\geq \frac{2k-i}{k-2}$. As $G[B]$ does not include a rainbow $K_k$, we can deduce $e(B)+c(B)\leq \binom{n-k}2+t_{n-k,k-2}+1$ based on Theorem~\ref{ExistClique}.
Let $s'$ and $i'$ be integers such that $n-k=s'(k-2)+i'$.Consequently, we obtain the following inequality
\begin{align}\label{e+c2}
n^2-\frac{(1+s)(n+i)}2+2\leq &
e(G)+c(G) \leq e(A)+c(A)+e(A,B)+c(A,B)+e(B)+c(B) \\
\leq & 2\binom k2 +(2k-2)(n-k)+(n-k)^2-\frac{(1+s')(n-k+i')}2+1.
\end{align}
There are two cases regarding the relationship between $s$ and $s'$, depending on the value of $i$:
\begin{enumerate}
\item If $i\leq 1$, then we have $s'=s-2$ and  $i'=i+k-4$. Substituting these values into the inequality~\ref{e+c2}, we have
$(s-1)k\leq 4s-3$;
\item If $2\leq i\leq k-3$, then we have $s'=s-1$ and $i'=i-2$. Substituting these values into the inequality~\ref{e+c2}, we have
$(s-1)k\leq 4s-i-1\leq 4s-3$.
\end{enumerate}

Both cases imply that
\begin{equation}
k\leq \frac{4s-3}{s-1}\leq5
\end{equation}
However, when $k=5$ and $s=2$, we have $n\geq 10$ and $i\geq 4$, which contradicts $i\leq k-3$. Consequently, we find that $k\leq 4$, leading to a contradiction.
\end{proof}

\section{Extremal Graphs without $K_4$ and $K_5$}

The following lemmas have motivations similar to those of Lemma 9, 10 and 11 in \cite{ehard2020rainbow}. Theorem \ref{main2} follows immediately from Lemma \ref{similarLemma9}, \ref{similarLemma10} and \ref{similarLemma11}.

\begin{lemma}\label{similarLemma9}
Let $G$ be a complete graph on $n$ vertices, edge-colored without a rainbow $K_k$, where $k=4$ and $n\geq 6$, or $k=5$ and $n\geq 9$. If $c(G)=t_{n,k-2}+1$ and $G$ contains a rainbow $T_{n,k-2}$ as a subgraph, then all edges within the partite sets of the rainbow $T_{n,k-2}$ have the same color, which is different from the colors used on the rainbow $T_{n,k-2}$.
\end{lemma}

\begin{proof} We divide the proof into two parts. First, when $k=4$ and $n\geq 6$, let $G$ be an edge-colored complete graph with $c(G)=t_{n,2}+1$. Suppose the colors $1,2,\ldots, t_{n,2}$ are used to color the rainbow subgraph $T_{n,2}$. Therefore, at least one edge in a partite set $V_1$ or $V_2$ is colored with a new color, say color 0. To complete the proof, we need to show that all the edges within each  partite set are colored with 0. Otherwise, there is a rainbow $K_4$. Note that in each partite set, there are at least 3 vertices, which we call them $v_1,v_2,v_3,\cdots, \in V_1$ and $v_1',v_2',v_3', \cdots, \in V_2$.

Without loss of generality, we can suppose the color of $v_iv_j'=3(i-1)+j$, where $i,j=1,2,3$. We can see that  $G[\{v_1,v_2,v_3\}]$  and $G[\{v_1',v_2',v_3'\}]$ are not rainbow triangles; otherwise, there would be a rainbow $K_4$. Since at least one edge in the partite $V_1$ and $V_2$ is colored by 0, let's assume $v_1v_2$ is colored by 0. Then the candidate colors for $v_1'v_2'$, such that there is no rainbow $K_4$, are $\{0,1,2,4,5\}$. By the same argument, we have $c(v_1'v_3')\in\{0,1,3,4,6\}$ and $c(v_2'v_3')\in\{0,2,3,5,6\}$.

Suppose $c(v_1'v_2')=0$, then $c(v_2v_3)\in\{0,4,5,7,8\}$ and $c(v_1v_3)\in\{0,1,2,7,8\}$. If $c(v_2v_3)=0$, then $c(v_2'v_3')\in\{0,2,3,5,6\}\cap\{0,5,6,8,9\} =\{0,5,6\}$ and $c(v_1'v_3')\in\{0,1,3,4,6\}\cap\{0,4,6,7,9\} =\{0,4,6\}$. If $c(v_2'v_3')=0$, then the candidate colors for $v_1v_3$ are $\{0,1,2,7,8\}\cap \{0,2,3,8,9\}=\{0,2,8\}$. If $c(v_1v_3)=0$, then the candidate colors of $v_1'v_3'$ are $\{0,1,3,7,9\}\cap\{0,4,6\}$. So the only possible color of $v_1'v_3'$ is 0.

So now let's go back one step. If $c(v_1v_3)=2$, then $c(v_1'v_3')\in\{2,1,3,7,9\}\cap \{0,4,6\}$, leading to a contradiction. The same argument goes with $c(v_1v_3)=8$.

Similarly, if $c(v_2'v_3')=5$, then $c(v_1v_3)\in\{0,1,2,7,8\}\cap\{5,2,3,8,9\}$. If $c(v_1v_3)=2$ or 8, we encounter  the same contradiction for $v_1'v_3'$. The same argument holds for $c(v_2'v_3')=6$.

If $c(v_2v_3)=4$, then $c(v_1v_3)\in\{0,4\}$ and
$c(v_2'v_3')\in\{0,2,3,5,6\}\cap\{4,5,6,8,9\}\cap\{0,4,2,3,8,9\}$, leading to a contradiction. Similarly, if  $c(v_2v_3)=5$, then $c(v_1v_3)\in\{0,5\}$ and $c(v_1'v_3')\in\{0,1,3,4,6\}\cap\{5,4,6,7,9\}\cap \{0,5,1,3,7,9\}$, leading to a contradiction.

If $c(v_2v_3)=7$, then $c(v_1v_3)\in\{0,7\}$ and  $c(v_2'v_3')\in\{0,2,3,5,6\}\cap\{7,5,6,8,9\}\cap\{0,7,2,3,8,9\}$, leading to a contradiction. Similarly, if  $c(v_2v_3)=8$, then $c(v_1v_3)\in\{0,8\}$, and $c(v_1'v_3')\in\{0,1,3,4,6\}\cap\{8,4,6,7,9\}\cap \{0,8,1,3,7,9\}$, leading to a contradiction.

If $c(v_1'v_2')\neq 0$, by symmetry, we can assume $c(v_1'v_2')=1$. Then $c(v_2v_3)\in\{1,4,5,7,8\}$.
If $c(v_2v_3)=1$, then $c(v_2'v_3')\in\{0,2,3,5,6\}\cap\{1,5,6,8,9\}=\{5,6\}$. If $c(v_2'v_3')=5$, then $c(v_1v_3)\in\{5,2,3,8,9\}$, which leads to a contradiction because $c(v_1v_3)$ can only be 0 or 1. The same argument applies when $c(v_2'v_3')=6$.

If $c(v_2v_3)=4$, then $c(v_1v_3)\in\{0,4\}$, and
$c(v_2'v_3')\in\{0,2,3,5,6\}\cap\{4,5,6,8,9\}\cap\{0,4,2,3,8,9\}$, which also leads to a contradiction.
The same argument holds for $c(v_2v_3)=5$, $7$ or $8$.

When $k=5$ and $n\geq 9$, let $G$ be an edge-colored complete graph with a rainbow $T_{n,3}$ as a subgraph, where each partite set contains at least $3$ vertices. Similar to the case when $k=4$, we denote the vertices in the three partite sets as follows:  $v_1^1,v_2^1,v_3^1, \ldots, \in V_1$, $v_1^2,v_2^2,v_3^2, \ldots, \in V_2$ and $v_1^3,v_2^3,v_3^3, \ldots, \in V_3$ with $n_1=|V_1|$, $n_2=|V_2|$ and $n_3=|V_3|$. Note that $c(T_{n,3})=t_{(n_1+n_2),2}+t_{(n_1+n_3),2}+t_{(n_2+n_3),2}$. We introduce an additional color, labeled as 0, which does not appear in the rainbow $T_{n,3}$. Suppose that at least one edge in $G$ is colored with $0$. Specifically, we assume that $c(v_1^1v_2^1)=0$, which implies that $c(G[V_1\cup V_j])\geq t_{(n_1+n_j),2}+1$ for $j=2$ or 3.

{\bf Claim} There is no rainbow $K_4$ in $G[V_1\cup V_j]$ for $j=2$ or 3.

{\it Proof of the Claim. }Let's begin by examining the subgraph $G[V_1\cup V_2]$. If there is a rainbow $K_4$ in $G[V_1\cup V_2]$, we have the following cases.

{\bf Case 1}: If the rainbow $K_4$ in $G[V_1\cup V_2]$ contains at most two colors that also appear in $G[V_1, V_3]$ or $G[V_2,V_3]$, the pigeonhole principle  guarantees the presence of a vertex in $V_3$ such that a rainbow $K_5$ can be formed by combining this vertex with the aforementioned rainbow $K_4$.

{\bf Case 2}: If at least three colors in this rainbow $K_4$ are present in the union of $G[V_1,V_3]$ and $G[V_2,V_3]$, then all of these edges must belong to $G[V_1]$. To illustrate, if one of them were in $G[V_2]$, we could combine it with $v^1_1v^1_2$ to construct a new rainbow $K_4$, satisfying Case 1. Consequently, all edges in $G[V_2]$ are colored with 0 or with colors that already appear in $G[V_1,V_2]$. Returning to the original rainbow $K_4$, we encounter two sub-cases:

{\it Sub-case 1. } If three vertices of the rainbow $K_4$ are in $V_1$ and the remaining vertex is in $V_2$, we consider the vertices as $v_i^1,v_j^1,v_k^1\in V_1$ and $v_{\ell}^2\in V_2$. Let's assume that the edges $v_i^1v_j^1$, $v_i^1v_k^1$, and $v_j^1v_k^1$ are colored with colors present in either $G[V_1,V_3]$ or $G[V_2,V_3]$. Now, we can choose any edge adjacent to $v_{\ell}^2$, say $v_{\ell}^2v_{p}^2$, and combine it with two of the vertices $v_i^1,v_j^1,v_k^1\in V_1$ to form a new rainbow $K_4$. This newly formed rainbow $K_4$ satisfies the conditions in Case 1.

{\it Sub-case 2. } If all vertices of the rainbow $K_4$ reside in $G[V_1]$ and if at least three edges have colors not present in $G[V_2]$, we can select an edge from $G[V_2]$. By combining this selected edge with one of the edges from the precious rainbow $K_4$, we create a new rainbow $K_4$, effectively returning us to Case 1.

In both of these cases, we are able to construct a rainbow $K_5$, implying that $c(G[V_1\cup V_2])=t_{(n_1+n_2),2}+1$.

We can apply a similar seasoning to the subgraph $G[V_1\cup V_3]$. {\it End of the Proof of the Claim.}

According to the Claim, the subgraph $G[V_1\cup V_j]$ forms a complete graph without a rainbow $K_4$, and $c(G[V_1\cup V_j])=t_{(n_1+n_j),2}+1$. Furthermore, it contains a rainbow $T_{(n_1+n_j),2}$, for $j=2$ and 3. By the first part of the proof for $k=4$, we conclude that the subgraphs  $G[V_i]$, with $i=1,2,3$, are colored uniformly with the same color, distinct from those used in the rainbow $T_{n,3}$.

\end{proof}

\begin{lemma}\label{similarLemma10}
Let $G$ be a complete graph with $n$ vertices, edge-colored with $t_{n,k-2}+1$ colors, where $k=4$ and $n\geq 8$, or $k=5$ and $n\geq 9$. If $G$ does not contain a rainbow $K_k$, then $G$ contains a rainbow $T_{n,k-2}$ as a subgraph.
\end{lemma}

\begin{proof}
We will prove the lemma by induction on $n$.
The lemma is true for the base cases where $k=4$ and $n=8$, as shown in Lemma~\ref{BaseCaseLemma10a}, and for $k=5$ and $n=9$, as shown in Lemma~\ref{BaseCaseLemma10b}. Now, let us assume that the lemma holds for graphs with fewer than $n$ vertices
(with $n\geq 9$ for $k=4$ and $n\geq 10$ for $k=5$).
From Lemma~\ref{ResultFromLemma11}, we ascertain the existence of a vertex $v\in V(G)$ with $d^s(v)=t_{n,k-2}-t_{n-1,k-2}$ and $c(G-v)=t_{n-1,k-2}+1$.
By the induction hypothesis, we know that $G-v$ contains a rainbow $T_{n-1,k-2}$.
Given that $n/(k-2)\geq 3$, according to Lemma~\ref{similarLemma9}, all edges within the partite sets of the rainbow $T_{n-1,k-2}$ share a common color distinct from those in the rainbow $T_{n-1,k-2}$. We can partition the vertices of  $G-v$ into $k-2$ partite sets, where edges within each partite set are colored by the same color. Subsequently, we reintroduce the vertex $v$ back into $G-v$. 
Furthermore, with the knowledge that $d^s(v)=t_{n,k-2}-t_{n-1,k-2}$, it follows that there exist $t_{n,k-2}-t_{n-1,k-2}$ edges colored distinctly by new colors, which are absent in $G-v$. As there is no rainbow $K_k$ in $G$, these $t_{n,k-2}-t_{n-1,k-2}$ edges must all be distributed to the $k-3$ larger partite sets. Consequently, $v$ must remain within the smallest partite sets, leading to the formation of a rainbow $T_{n,k-2}$. Thus, the proof is complete.
\end{proof}

\begin{lemma}\label{similarLemma11}
If $G$ is an edge-colored graph on $n$ vertices such that $e(G)+c(G)={n \choose 2}+t_{n,k-2}+1$, where $k=4$ and $n\geq 8$ (or $n\geq 9$ if $k=5$). If $G$ does not contain a rainbow $K_k$, then $G$ is complete.
\end{lemma}

\begin{proof}
For $n\geq 8$, $k=4$ (or $n\geq 9$ if $k=5$), we prove the lemma by induction on $n$.
For $k=4$ and $n=8$, or $k=5$ and $n=9$, we have the conclusion by Lemma~\ref{BaseCaseLemma11a} and Lemma~\ref{BaseCaseLemma11b}.
Suppose that the lemma is true for graphs with fewer than $n$ vertices ($n\geq 9$ if $k=4$ and $n\geq 10$ if $k=5$).
By Lemma~\ref{ResultFromLemma11}, there is a vertex $v$ in $G$ such that $e(G-v)+c(G-v)={n-1 \choose 2}+t_{n-1,k-2}+1$. Therefore, we have
\begin{align}\label{dds}
d(v)+d^s(v)=n-1+t_{n,k-2}-t_{n-1,k-2}.
\end{align}
By the induction hypothesis, $G-v$ is complete. 
Using Lemma~\ref{similarLemma9} and Lemma~\ref{similarLemma10},
we conclude that $G-v$ contains a rainbow $T_{n-1,k-2}$, and all edges within the partite sets of the rainbow $T_{n-1,k-2}$ are colored with the same color which is not used on the rainbow $T_{n-1,k-2}$.
If $d^s(v)\geq t_{n,k-2}-t_{n-1,k-2}+1$, then $v$ connects to all partite sets of the rainbow $T_{n-1,k-2}$ by the pigeonhole principle.
By taking one vertex adjacent to $v$ from each partite set,
along with an extra vertex from the rainbow $T_{n-1,k-2}$ (the existence is guaranteed by the size of each partite set), we can obtain a rainbow $K_k$, which contradicts the assumption.
Thus, we have $d^s(v)\leq t_{n,k-2}-t_{n-1,k-2}$.
Combining this inequality with Equation~\ref{dds}, we conclude that $d(v)=n-1$, and therefore $G$ is complete.
The proof is finished.
\end{proof}

\begin{Remark}
Together with Remark \ref{CounterExample} and the example in the concluding remark of \cite{ehard2020rainbow}, our result in Theorem \ref{main2} is tight.
\end{Remark}

\section*{ACKNOWLEDGMENTS}
We thank Maurício Collares Neto for useful discussions. 
Danni Peng is supported by CNPq (Grant Number  141537/2023-0).  CNPq is the national Council for Scientific and Technological Development of Brazil. 

\section*{Appendix}\label{appendix}
\begin{proof}[Proof of Lemma \ref{noK3}]
We will prove by using induction on the number of vertices of $G$.  When $n=3$,  $e(G)+c(G)=\frac{3\cdot 4}{2}-1=5$. Clearly, it follows that $e(G)=3$ and then $c(G)=2$. Then $G=K_3$ and $c(i,j)=i$, if $i<j$.

Now suppose the conclusion is true for $n-1$. When $G$ has $n$ vertices, we need the following claim.

{\bf Claim.} Denote by $d^s(v)$ the number of saturated colors at $v$, there exists a vertex $v$ such that $d(v)+d^s(v)\leq n$.
\begin{proof}
If not, then
\begin{equation}\label{claim}
\sum_{v\in V(G)}(d(v)+d^s(v))\geq n(n+1)
\end{equation}
 Now, for any color, it is saturated by less than three vertices. So $\sum_{v\in V(G)}d^s(v)\leq 2c(G)$. Therefore, the left hand of \ref{claim} satisfies
\begin{equation}
\sum_{v\in V(G)}(d(v)+d^s(v))\leq 2e(G)+2c(G)=n(n+1)-2
\end{equation}
which is a contradiction.
\end{proof}
Following the claim, we can find a vertex $v$ and let $G'=G-v$. Then
\begin{equation}
e(G')+c(G')\geq \frac{n(n+1)}{2}-1-n=\frac{n(n-1)}{2}-1
\end{equation}
Since $G'$ has no rainbow triangle, we have $e(G')+c(G')=\frac{n(n-1)}{2}-1$. By induction hypothesis, we have $G'$ is $K_{n-1}$, with $c(i,j)=i$, when $i<j$, $i,j=1,\cdots n-1$.  Let $V(G')=\{v_1,v_2\ldots, v_{n-1}\}$.

Now we add the vertex $v$ back to $G'$. Note that $d(v)+d^s(v)=n$. However, $d^s(v)$ cannot be larger than 1, because if two new colors are saturated at $v$, we can construct a rainbow triangle by choosing one edge in each color and an edge between them. Therefore, $d^s(v)=1$ and $d(v)=n-1$, which proves that $G=K_n$.  Let {\it red} be the new color saturated at $v$.

Suppose the red star has $k$ leaves. If $k=1$, say, $vv_{\ell}$ is red, we can assign the colors for the rest of edge as  $c(vv_{\ell'}=\ell'$, for $\ell'<l$ and $c(vv_{\ell'})=\ell$. Clearly, there is no rainbow triangle and by a simple rename of the vertices, it can be seen that the color function holds immediately.

For the case $k>1$, we need the following claim.

{\bf Claim} If $k>1$, $vv_{n-1}, vv_{n-2},\ldots,vv_{n-k}$ are red.

\begin{proof}
Suppose $vv_p$ and $vv_q$, $p<q$, are red and $vv_{p+1}$ is not red. To make sure that $vv_pv_{p+1}$ is not a rainbow triangle, $c(vv_{p+1})=c(v_pv_{p+1})=p$. On the other side, $vv_{p+1}v_{q}$ is not a rainbow triangle. $c(vv_{p+1})=c(v_{p+1}v_q)=p+1$, which gives a contradiction. With similar argument, $vv_{n-1}$ is red because otherwise $vv_pv_{n-1}$ and $vv_qv_{n-1}$ are rainbow triangles. So  $vv_{n-1}, vv_{n-2},\ldots,vv_{n-k}$ are red.
\end{proof}

Following the claim, we assign $c(v_iv)=i$, $i<n-k$ and finish the proof.
\end{proof}

\bibliographystyle{plain}
\addcontentsline{toc}{chapter}{Bibliography}
\bibliography{rainbow_clique}
\end{document}